\documentclass[11pt]{article}
\usepackage{amsfonts}
\usepackage{amsmath}
\usepackage{amssymb}
\usepackage{amsthm}
\usepackage{stmaryrd}
\usepackage{stix}
\usepackage[all]{xy}
\usepackage{tikz}
\usepackage{tikz-cd}
\tikzset{% for drawing adjunctions
    symbol/.style={%
        draw=none,
        every to/.append style={%
            edge node={node [sloped, allow upside down, auto=false]{$#1$}}}
    }
}

\usepackage{hyperref}
\hypersetup{
	linktocpage,
    colorlinks = True,
    citecolor=gray,
    filecolor=black,
    linkcolor=blue,
    urlcolor=blue
}

\newcommand{\colim}{\mathrm{colim}}
\newcommand{\ints}{\mathbb{Z}}

\newcommand{\sph}{\mathbb{S}}
\newcommand{\E}{\mathbb{E}}
\newcommand{\modd}{/\!\! /}
\newcommand{\A}{\mathbb{A}}

\newcommand{\C}{\mathcal{C}}

\theoremstyle{definition} 
\newtheorem{rem}{Remark}

\newtheorem{defi}{Definition}

\theoremstyle{plain}
\newtheorem{thm}{Theorem}
\newtheorem{lem}{Lemma}
\newtheorem*{lem*}{Lemma}
\newtheorem{prop}{Proposition}
\newtheorem*{prop*}{Proposition}
\newtheorem*{thm*}{Theorem}
\newtheorem*{cor*}{Corollary}
\newtheorem{cor}{Corollary}

\begin{document}

\title{A Theorem on Multiplicative Cell Attachments with an Application to Ravenel's $X(n)$ Spectra}
\author{Jonathan Beardsley}
\maketitle

\begin{abstract}
	We show that the homotopy groups of a connective $\E_k$-ring spectrum with an $\E_k$-cell attached along a class $\alpha$ in degree $n$ are isomorphic to the homotopy groups of the cofiber of the self-map associated to $\alpha$ through degree $2n$. Using this, we prove that the $2n-1^{st}$ homotopy groups of Ravenel's $X(n)$ spectra are cyclic for all $n$. This further implies that, after localizing at a prime, $X(n+1)$ is homotopically unique as the $\E_1$-$X(n)$-algebra with homotopy groups in degree $2n-1$ killed by an $\E_1$-cell. Lastly, we prove analogous theorems for a sequence of $\E_k$-ring Thom spectra, for each odd $k$, which are formally similar to Ravenel's $X(n)$ spectra and whose colimit is also $MU$.  
\end{abstract}

\section{Introduction}\label{intro}

This paper extends the work of understanding the spectra $X(n)$ begun in \cite{beardsrelative}. These spectra were introduced by Ravenel in \cite{rav} and were used in an essential way in Devinatz, Hopkins and Smith's proof of Ravenel's Nilpotence Conjecture \cite{devhopsmith}. These spectra are similar to the complex cobordism spectrum $MU$ in a number of ways. Perhaps most importantly, much of the theory of \textit{complex orientations} for ring spectra can be developed for $X(n)$ by replacing $\mathbb{C}P^\infty$ with $\mathbb{C}P^{n-1}$ and formal group laws with \textit{formal group law $n$-buds} (cf.~\cite[Ch. II]{laz}). The interested reader can find a thorough description of the theory of complex orientations, both for $MU$ and $X(n)$, in \cite[\S 4.1, \S 6.5]{rav}. The work contained in this paper, along with \cite{beardsrelative}, also represents a small step in understanding chromatic homotopy theory and its fundamental objects, e.g.~$\sph$, $MU$ and $X(n)$, via the derived algebraic geometry framework developed in \cite{ha}. 

The main technical development of this paper is Theorem \ref{thm:presurvive} in Section \ref{section:construction}:
\begin{thm*}
	Let $A$ be a connective $\E_{k+1}$-ring spectrum and choose an element $\alpha\in\pi_d(A)$ for $d\geq 0$ with associated self-map $\bar{\alpha}\colon \Sigma^dA\to A$.  Then there is a map of $A$-modules $cof(\bar{\alpha})\to A\modd_{\E_k}\alpha$ which induces an isomorphism on homotopy groups in degrees less than $2d$. 
\end{thm*}
\noindent Here, the spectrum $A\modd_{\E_k}\alpha$ is the \textit{versal $\E_k$-algebra on $A$ of characteristic $\alpha$}, described for $k=\infty$ in \cite{szymikprimechar} and for all $k$ in \cite{acb}. It is constructed as by ``attaching a $d$-cell'' to $A$ along $\alpha$ in the category of $\E_k$-$A$-algebras. Often, versal algebras of this form are conceptually interesting but have difficult to determine homotopy groups, so the above theorem may be of interest beyond its uses in this paper.

In Section \ref{xnspecif} we apply the above theorem to prove that $\pi_{2n-1}(X(n))$ is cyclic, and that, moreover, it is generated by a certain element $\chi_n$ which was identified in \cite{beardsrelative}. This is Corollary \ref{generators}:

\begin{cor*}
	The element $\chi_n\in\pi_{2n-1}(X(n))$ generates $\pi_{2n-1}(X(n))$ for all $n\geq 1$. 
\end{cor*}

\noindent These elements $\chi_n$ arise from the facts that $X(n)$ is the Thom spectrum of a map $\Omega SU(n)\to BGL_1(\sph)$ and that $\Omega SU(n)$ sits in a fiber sequence $\Omega SU(n)\to \Omega SU(n+1)\to \Omega S^{2n-1}$. In \cite{beardsrelative} it was shown that $X(n)$ is a quotient of $\sph$ by an action of $\Omega^2 SU(n)$. Because the action of $\Omega^2 SU(n)$ on $\sph$ factors through $\Omega^2 SU(n+1)$, there is a residual action of $\Omega^2 SU(n+1)/\Omega^2 SU(n)\simeq\Omega^2 S^{2n-1}$ on $\sph/\Omega^2 SU(n)\simeq X(n)$. This yields an $\E_1$-map $\Omega S^{2n-1}\to BGL_1(\sph/\Omega^2 SU(n-1))\simeq BGL_1(X(n))$. This is equivalent data to a map of pointed spaces $S^{2n-2}\to BGL_1(X(n))$, and we take $\chi_n$ to be the element canonically associated to it by \cite[Definition 4.8]{acb}. 

It should be remarked that very little is known about the homotopy groups of $X(n)$ in general. A complete knowledge of their homotopy groups is at least as difficult to obtain as complete knowledge of the stable homotopy groups of spheres. However, some computational progress has been made by Ravenel and others, and the state of the art can be found in \cite{nakairavenel}.

While the above allows us to identify $\chi_n$ as a generator of $\pi_{2n-1}(X(n))$, and $X(n+1)$ as the versal $\E_1$-$X(n)$-algebra of characteristic $\chi_n$, this still relies on the somewhat geometric construction of $\chi_n$ given in \cite{beardsrelative}. However, after localizing at a prime $p$, we see in Theorem \ref{thm:canonical} that simply attaching an $\E_1$-cell to $X(n)$ along \textit{any} generator of $\pi_{2n-1}$ will still produce (the $p$-localization of) $X(n+1)$:

\begin{thm*}
	The spectrum $X(n+1)_{(p)}$ can be obtained from $X(n)_{(p)}$ by attaching an $\E_1$-$X(n)_{(p)}$-cell along any generator of the cyclic group $\pi_{2n-1}(X(n)_{(p)})$.
\end{thm*}
\noindent We make special note of this only because of its similarity to the work of Priddy on constructing the Brown-Peterson spectrum $BP$, which is a summand of $MU_{(p)}$, by coning off all elements in odd degrees \cite{priddybp}. The above theorem suggests that $X(n+1)_{(p)}$ is the spectrum obtained when one cones off the homotopy of $X(n)_{(p)}$ in degree $2n-1$ using an $\E_1$-cell (and that, in some poorly defined sense, $MU_{(p)}$ is obtained by coning off all homotopy in odd degrees, but with $\E_1$-cells). For the time being, this similarity appears to be entirely superficial. We still include the theorem however as it, at the very least, hints at a purely algebraic construction of $MU$, as requested by Ravenel in \cite[\S 1.3]{rav}.

Finally, in Section \ref{section:xnk} we show that for any odd integer $k$ there is a sequence of $\E_k$-ring spectra $\{X(n,k)\}_{n\geq 0}$ with $X(0,k)\simeq \sph$ whose colimit over $n$ is equivalent to $MU$. The spectrum $X(n,k)$ is the Thom spectrum of the base point component of the map $\Omega^k SU(n)\hookrightarrow \Omega^k SU\simeq \ints\times BU$. We prove that these spectra satisfy similar properties to the $X(n)$ spectra. In particular we have Lemma \ref{lem:equivbelow} which completely describes their homotopy in low degrees
\begin{lem*}
	For odd $k$, the map $X(n,k)\to MU$ induced by the inclusion $\Omega^k SU(n)\hookrightarrow BU$ induces an isomorphism on homotopy in degrees less than or equal to $2n-k-1$. 
\end{lem*}
\noindent which generalizes the analogous statement for $X(n,1)=X(n)$ in \cite[\S 6.5]{rav}. We also formulate suitably adjusted versions of Corollary \ref{generators} and Theorem \ref{thm:canonical} for the spectra $X(n,k)$. 

These spectra may prove useful in understanding chromatic homotopy theory and the nilpotence detecting properties of $MU$, but because of the relative complexity of even their homology groups (as they are isomorphic to those of $\Omega^k SU(n))$, we postpone a thorough investigation to future work. 

There is also an appendix, Section \ref{appendix}, in which we prove several statements of a more category theoretic nature. These statements will be well known to experts. However, rigorous proofs of them in the language of quasicategories do not seem to be available in the literature, so we include them.

\subsection{Notation and Terminology}

%In Section \ref{section:construction}, we review the concept of ``attaching an $\E_k$-cell'' and give several technical lemmas. The main theorem of this section proves an intuitively obvious fact: if one attaches an $\E_k$-$A$-cell to $A$ along $\alpha$ in degree $d>0$, and $A$ is connective, then this cell attachment won't kill off classes in $\pi_d(A)$ \emph{other} than $\alpha$. Section \ref{xnspecif} then contains the theorem which gives a new construction of $MU$ by iterated attachment of $\A_\infty$-cells. The final section of this note gives an entire family of constructions of $MU$ as $\E_k$-cell attachments for any odd $k$. Unfortunately in these more structured cases we know next to nothing about the spectra interpolating between $\sph$ and $MU$.

All of what follows uses the quasicategorical and derived algebra machinery of Lurie, as described in \cite{htt} and \cite{ha}. Thus all of our categories will be quasicategories and any time we refer to a limit or colimit we will mean the quasicategorical limit or colimit as defined in \cite{htt}. However, understanding the main statements in this paper should not require much specialized knowledge of Lurie's work. We work specifically in the quasicategories of pointed Kan complexes (which we will refer to as spaces) and spectra. For both we will denote the tensor product by $\wedge$. We also use the shorthands $[n]=\{0,1,\ldots,n\}$ and $\langle n\rangle=\{\ast,1,\ldots,n\}$, where the latter is thought of as a \emph{pointed} set. It seems unlikely that the theory of quasicategories is necessary to proving the results of this note. However, we make heavy use of results from \cite{acb} and \cite{beardsrelative}, both of which use that language.

Most notation is described as it used in the text, but we will set a few standards here. By saying that a space or spectrum is \textit{$n$-connective}, we mean that it has trivial homotopy groups \textit{below degree $n$}. This should be contrasted with saying that an object is $n$-\textit{connected}, which means that it has trivial homotopy groups in degrees \textit{less than or equal to $n$}. For example, any space $X$ is $0$-connective, but is only $0$-connected if $\pi_0(X)\cong 0$, and would be called $2$-connective if it were ``simply connected'' in the classical sense. When a spectrum is $0$-connective, we will just say that it is connective.  A morphism of spaces or spectra will be called \textit{$n$-connective} if it induces an isomorphism of homotopy groups in degress less than $n$. 

We work throughout with the \textit{little $k$-disks operads} $\E_k$ for $0\leq k\leq \infty$ and freely use the fact that $\E_1\simeq \A_\infty$. We also remind the reader that an $\E_0$-algebra in a monoidal quasicategory is nothing more than an object equipped with a map from the monoidal unit. Given an $\E_{k+1}$-ring spectrum $R$, there is an  $\E_{k}$-monoidal quasicategory of modules over $R$ which we denote $LMod_{R}$. There is also a quasicategory of $\E_k$-algebras in $LMod_{R}$ which we denote by $Alg_{R}^{\E_k}$. Given an $\E_k$-ring spectrum $R$ and $0\leq j<k$, there always exists an $\E_j$-ring spectrum which is equivalent to $R$ as a spectrum. When we wish to talk about this spectrum, we will simply say ``$R$ thought of as an $\E_j$-ring spectrum,'' or ``...as an $\E_j$-algebra.'' The above construction induces a forgetful functor $U_{\E_k}^{\E_j}\colon Alg_{\sph}^{\E_k}\to Alg_{\sph}^{\E_j}$. This functor preserves limits and so has a colimit preserving left adjoint which we will denote by $F_{\E_j}^{\E_k}$. Given a fixed $\E_k$-ring spectrum $R$, an $\E_j$-$R$-algebra $A$ for some $0<j<k$, and an integer $0<l<j$,  we will denote the ``free $\E_l$-$A$-algebra on an $A$-module $M$'' by $F^{\E_l}(M)$, where the empty subscript indicates that $M$ has no operadic structure. This also applies to the case $k=j=\infty$.

\subsection{Acknowledgements}

Finally, for this author at least, mathematics is a deeply social and communal pursuit. Many people, through e-mails, MathOverflow, and personal conversations, were helpful in understanding much of what follows. Some, but likely not all, of those people are: Bob Bruner, Andrew Salch, Nicolas Ricka, Gabriel Angelini-Knoll, Bogdan Gheorghe, Eric Peterson, Sean Tilson, Omar Antolin-Camarena, Tobias Barthel, Tom Bachmann, Tyler Lawson and Marc Hoyois. The work presented here also benefited \textit{immensely} from the comments of an anonymous referee.

\section{The Effect of Attaching a Structured Cell}\label{section:construction}

In what follows, given an element $\alpha\in\pi_d(A)$ for an $\E_{k+1}$-ring spectrum $A$, the phrase \emph{attach an $\E_k$-$A$-cell along $\alpha$} means to produce the $\E_k$-$A$-algebra $A\modd_{\E_k}\alpha$ defined by the following pushout in $\E_k$-$A$-algebras: 

$$\xymatrix{
F^{\E_k}(\Sigma^{d}A)\ar[r]^-{adj(0)}\ar[d]^-{adj(\bar{\alpha})}&A\ar[d]\\
A\ar[r]& A\modd_{\E_k}\alpha.
}$$

\noindent In the above diagram, $\bar{\alpha}$ is the morphism $\Sigma^dA\to A\wedge A\to A$ induced by $\alpha\colon\sph^d\to A$ and $adj(-)$ is the adjunction equivalence $LMod_A(\Sigma^d A,A)\overset{\sim}\to Alg_A^{\E_k}(F^{\E_k}(\Sigma^d A), A)$. The $\E_k$-$A$-algebra $A\modd_{\E_k}\alpha$ is called the ``versal $\E_k$-$A$-algebra of characteristic $\alpha$'' in \cite{acb} because it always admits a not-necessarily-unique map to any other $\E_k$-$A$-algebra on which multiplication by $\alpha$ is nullhomotopic.

Our goal is to prove a lemma about elements in homotopy that always survive attaching $\E_k$-cells. Intuitively what we will prove is that, given a connective $\E_{k+1}$-ring spectrum $A$ and $\alpha\in\pi_d(A)$ with induced self-map $\bar{\alpha}\colon\Sigma^dA\to A$, the spectrum $A\modd_{\E_k}\alpha$ has a filtration whose associated graded $A$-module is $cof(\bar{\alpha})\oplus D$ where $D$ is $2d$-connective. In particular, there is a map $cof(\bar{\alpha})\to A\modd_{\E_k}\alpha$ which is an isomorphism in $\pi_i$ for $i<2d$.

The main ideas behind the proof of the following Proposition \ref{prop:connofFM} were explained to me by Tyler Lawson. We remind the reader that an $\E_0$-$A$-algebra is an $A$-module $M$ equipped with a unit map $A\to M$. We also point out that we use $\oplus$ rather than $\coprod$ to denote the coproduct in the category $LMod_A$. 

\begin{prop}\label{prop:connofFM}
Let $A$ be an $\E_{k+1}$-ring spectrum and $F_{\E_0}^{\E_k}(M)$ be the free $\E_k$-algebra on an $\E_0$-$A$-algebra $M$ with unit map $u\colon A\to M$. If $u$ induces an isomorphism on homotopy in degrees less than or equal to $d$, then the universal map $\eta\colon M\to F_{\E_0}^{\E_k}(M)$ induced by the free forgetful adjunction is and isomorphism on homotopy in degrees less than or equal to $2d$.
\end{prop}

\begin{proof}
	Let $F^{\E_k}(M)$ be the free $\E_k$-algebra on the $R$-module $M$ (without its $\E_0$-structure). Note that $F^{\E_k}(M)\simeq F_{\E_0}^{\E_k}(A\oplus M)$ since $A\oplus M$ is the free $\E_0$-algebra on $M$ and $F^{\E_k}\circ F^{\E_0}\simeq F^{\E_k}$. Thus we have an equivalence (of underlying $A$-modules) $F_{\E_0}^{\E_k}(A\oplus M)\simeq \bigoplus_{n\geq 0}Sym_{\E_k}^n(M)$. As such, by passing along this equivalence, there is an exhaustive filtration of $F_{\E_0}^{\E_k}(A\oplus M)$ whose filtration quotients are equivalent to $Sym_{\E_k}^n(M)$. In other words, the associated graded of this filtration is equivalent to $F^{\E_k}(M)$. Notice also that, since $\oplus$ is the biproduct in $LMod_A$, a morphism $A\oplus X\to A\oplus Y$ is entirely determined by a pair of maps $X\to A$ and $X\to Y$, so the induced map $\bigoplus_{n\geq 0}Sym_{\E_k}^n(X)\to \bigoplus_{n\geq 0}Sym_{\E_k}^n(Y)$ must respect the filtration induced by the grading. Thus we have filtration preserving morphisms $F^{\E_k}(X)\to F^{\E_k}(Y)$ whenever we have maps $A\oplus X\to A\oplus Y$. 
	
	Noticing that the pushout of the diagram in $\E_0$-$A$-algebras (which can be computed in $A$-modules) $A\overset{1_A}{\leftarrow}A\overset{u}\to M$ must be equivalent to $M$, we can use Lemma \ref{barkan} to write $M$ as the colimit of the simplicial object $\{A\oplus M\leftleftarrows A\oplus A\oplus M\leftthreearrows\cdots\}$ with $n^{th}$ term equivalent to $A\oplus A^{\oplus n}\oplus M$. As a left adjoint, the functor $F_{\E_0}^{\E_k}$ commutes with colimits (where we are implicitly using that colimits of $\E_0$-algebras are computed in $LMod_A$), so we may write:
	\begin{align*}
	F_{\E_0}^{\E_k}(M)&\simeq colim(F_{\E_0}^{\E_k}(A\oplus M)\leftleftarrows F_{\E_0}^{\E_k}(A\oplus A\oplus M)\leftthreearrows\cdots)\\
	&\simeq colim(F^{\E_k}(M)\leftleftarrows F^{\E_k}(A\oplus M)\leftthreearrows\cdots).
	\end{align*}
	
	Now, applying our filtration levelwise and using the fact that the induced maps respect filtration, we have a filtered simplicial object with filtration quotients:
	\[
	Sym_{\E_k}^{n}(M)\leftleftarrows Sym_{\E_k}^n(A\oplus M)\leftthreearrows\cdots.
	\]
	
	Since each functor $Sym^n_{\E_k}$ preserves sifted colimits (of which simplicial resolutions are an example) by Lemma \ref{lemma:tpcommuteswithsifted}, we have that the filtration quotients of the induced filtration on $F_{\E_0}^{\E_k}(M)$ are equivalent to $Sym_{\E_k}^n(cof(u))$, i.e.~the associated graded of the induced filtration on $F_{\E_0}^{\E_k}(M)$ is $\bigoplus_{n\geq 0} Sym_{\E_k}^n(cof(u))$. Here we are using the fact that the colimit of the simplicial object $\{M\leftleftarrows A\oplus M\leftthreearrows\cdots\}$ is equivalent to (again by Lemma \ref{barkan}) the pushout of the diagram $0\overset{0}\leftarrow A\overset{u}\to M$, i.e.~the cofiber of $u$.
	
	Examining the filtration on this simplicial object more closely, and noticing that it arises from the ``symmetric powers'' filtration of $F^{\E_k}(A\oplus A^{\oplus n}\oplus M)$ at each simplicial degree, we see that its first two levels are the constant simplicial object:
	\[
	A\leftleftarrows A\leftthreearrows \cdots,
	\]
	and the constant simplicial object summed with the simplicial object we had before applying $F_{\E_0}^{\E_k}$, which has the form:
	\[
	A\oplus M\leftleftarrows A\oplus A\oplus M\leftthreearrows \cdots.
	\]
	These two objects have colimit $A$ and $M$ respectively. Thus we know that our filtration of $F_{\E_0}^{\E_k}(M)$ begins as $A\overset{u}\to M$. We know that the map from $A$ to $M$ is the unit map $u$ because $F_{\E_0}^{\E_k}(M)$ must have the same unit as $M$ (this also corresponds to the fact that the first filtration quotient must be equivalent to $cof(u)$). It is also clear that the filtration structure map $M\to F_{\E_0}^{\E_k}(M)$ must be the monadic unit map $\eta$. Hence the cofiber of the unit $\eta\colon M\to F_{\E_0}^{\E_k}(M)$ has a filtration with associated graded $\bigoplus_{n\geq 2} Sym_{\E_k}^n(cof(u))$. Thus the connectivity estimates of Lemma \ref{lemma:symconnected} give the result. 
\end{proof}

\begin{rem}
	There is nothing particular to the category $LMod_A$ that is necessary for the above constructions other than being a stable, cocomplete quasicategory and having a suitable notion of $n$-connective objects (i.e.~a $t$-structure). Thus the above is a relatively general description of the functor $F_{\E_0}^{\E_k}$ for any $k\geq 0$. 
\end{rem}

\begin{thm}\label{thm:presurvive}
	Let $A$ be a connective $\E_{k+1}$-ring spectrum and choose an element $\alpha\in\pi_d(A)$ for $d\geq 0$.  Then there is a map of $A$-modules $cof(\bar{\alpha})\to A\modd_{\E_k}\alpha$ which is $2d$-connective. 
\end{thm}

\begin{proof}
	By \cite[Lemma 4.4]{acb} we have that $A\modd_{\E_k}\alpha\simeq F_{\E_0}^{\E_k}(cof(\bar{\alpha}))$. Since $cof(\bar{\alpha})$ is the cofiber of a map $\bar{\alpha}\colon\Sigma^dA\to A$ which is $d$-connective, the unit map $A\to cof(\bar{\alpha})$ is also $d$-connective. The result then follows from Proposition \ref{prop:connofFM}. 
\end{proof}

\begin{cor}\label{survive}
	If $A$ is an $\E_{k+1}$-ring spectrum with $\alpha\in\pi_d(A)$ and $\pi_{d}(A\modd_{\E_k}\alpha)\cong 0$ then $\alpha$ generates $\pi_d(A)$ as a $\pi_0(A)$-module.  
\end{cor}

\begin{proof}
	By Theorem \ref{thm:presurvive} there is a $2d$-connective map of $A$-modules $cof(\bar{\alpha})\to A\modd_{\E_k}\alpha$. By considering the long exact sequence in homotopy for the cofiber sequence $\Sigma^dA\overset{\bar{\alpha}}\to A\to cof(\bar{\alpha})$, we see that we have an exact sequence of $\pi_0(A)$-modules $$\pi_d(\Sigma^dA)\cong\pi_0(A)\overset{-\cdot\alpha}\to \pi_d(A)\to \pi_d(cof(\bar{\alpha}))\to\pi_{d-1}(\Sigma^dA)\cong 0.$$ From this it follows that $\pi_d(A\modd_{\E_k}\alpha)\cong\pi_d(cof(\bar{\alpha}))\cong \pi_d(A)/[\alpha]\pi_0(A)$. Thus if $\pi_d(A\modd_{\E_k}\alpha)\cong 0$ it must be the case that $\pi_d(A)$ is generated by $[\alpha]$ as a $\pi_0(A)$-module. 
\end{proof}

\section{$\pi_{2n-1}(X(n))$ and the Construction of $X(n)$ by Attaching a Cell}\label{xnspecif}

Corollary \ref{survive} immediately applies to $\pi_{2n-1}$ of the spectra $X(n)$:

\begin{cor}\label{generators}
The element $\chi_n\in\pi_{2n-1}(X(n))$, from \cite[Corollary 13]{beardsrelative}, generates $\pi_{2n-1}(X(n))$ for all $n\geq 1$. 
\end{cor}

\begin{proof}
From \cite[Corollary 13]{beardsrelative} we know that $X(n+1)\simeq X(n)\modd_{\E_1}\chi_n$. From the paragraph following Proposition 6.5.4 of \cite{rav} (or by noticing that $X(n+1)$ is the Thom spectrum of $\Omega SU(n+1)\to BU$, which is an equivalence below degree $2n+1$), we have that $\pi_{2n-1}(X(n+1))\cong\pi_{2n-1}(MU)=0$. Hence, by Corollary \ref{survive}, $\chi_n$ must generate $\pi_{2n-1}(X(n))$. 
\end{proof}

Furthermore, if we localize at a prime, we need not have \emph{known} about the elements $\chi_n$ at all. In other words, we can attach an $\E_1$-cell along any generator of $\pi_{2n-1}$ and still produce (the $p$-localization of) $X(n+1)$. 

\begin{thm}\label{thm:canonical}
The spectrum $X(n+1)_{(p)}$ can be obtained from $X(n)_{(p)}$ by attaching an $\E_1$-$X(n)_{(p)}$-cell along any generator of the cyclic group $\pi_{2n-1}(X(n)_{(p)})$.
\end{thm}

\begin{proof}
Let $\chi_{n,p}$ be the image of $\chi_n$ under the $p$-localization map $\pi_{2n-1}(X(n))\to\pi_{2n-1}(X(n)_{(p)})\cong \pi_{2n-1}(X(n))_{(p)}$. From  \cite[Proposition 2.2.1.9]{ha} we know that the $p$-localization functor is symmetric monoidal, so preserves $\E_n$-algebras for all $n$. Additionally, from \cite[Remark 7.3.2.13]{ha}, and the results referenced in that remark, we know that this localization functor lifts to a left adjoint on categories of $\E_n$-algebras, so in particular preserves colimits of $\E_1$-algebras (for the analogous theorem in model categories, see \cite[Theorem 5.6]{batwhite}). Thus we know that cell attachment commutes with localization, so $X(n+1)_{(p)}$ is equivalent to $X(n)_{(p)}$ with an $\E_1$-$X(n)_{(p)}$-cell attached along $\chi_{n,p}$. 

Recall from \cite[Lemma 4.4]{acb} that the versal $\E_n$-algebra on a class $\alpha$ is equivalent to the free $\E_n$-algebra on the $\E_0$-algebra $cof(\bar{\alpha})$. This implies that $X(n+1)\simeq F_{\E_0}^{\E_1}(cof(\bar{\chi}_n))$. Since localization, as a left adjoint, commutes with free functors, and preserves algebraic structure as described above,  we can describe $X(n+1)_{(p)}$ as the free $\E_1$-$X(n)_{(p)}$-algebra induced up from the $\E_0$-$X(n)_{(p)}$-algebra $cof(\bar{\chi}_{n,p})$.

We now have that $\chi_{n,p}$ generates $\pi_{2n-1}(X(n)_{(p)})$ as a finite module over the local ring $\pi_0(X(n)_{(p)})\cong\mathbb{Z}_{(p)}$. Any finite cyclic $\mathbb{Z}_{(p)}$-module is isomorphic to $\mathbb{Z}/p^e$ for some $e$, and an element in $\mathbb{Z}_{(p)}$ represents a generator if and only if it is not divisible by $p$. But these are precisely the units of $\mathbb{Z}_{(p)}$, so any other generator must differ from $\chi_{n,p}$ by a unit of $\mathbb{Z}_{(p)}$. If $\alpha$ is another generator, call the associated unit $u_\alpha\in\pi_{0}(X(n)_{(p)})\cong \ints_{(p)}$. 

We have a commutative diagram whose rows are cofiber sequences:

\begin{center}
\begin{tabular}{c}

\xymatrix{
\Sigma^{2n-1}X(n)_{(p)}\ar@{=}[d]\ar[r]^{~~~~\bar{\chi}_{n,p}} & X(n)_{(p)}\ar[d]_{u_\alpha}\ar[r]& cof(\bar{\chi}_{n,p})\\
\Sigma^{2n-1}X(n)_{(p)}\ar[r]^{~~~~\bar{\alpha}} & X(n)_{(p)}\ar[r]& cof(\bar{\alpha}).
}
\end{tabular}
\end{center}

The universal property of the cofiber induces a map $\phi$ filling in the diagram: 

\begin{center}
\begin{tabular}{c}

\xymatrix{
\Sigma^{2n-1}X(n)_{(p)}\ar@{=}[d]\ar[r]^{~~~~\bar{\chi}_{n,p}} & X(n)_{(p)}\ar[d]_{u_\alpha}\ar[r]& cof(\bar{\chi}_{n,p})\ar[d]^{\phi}\\
\Sigma^{2n-1}X(n)_{(p)}\ar[r]^{~~~~\bar{\alpha}} & X(n)_{(p)}\ar[r]& cof(\bar{\alpha})
}
\end{tabular}
\end{center}
and an application of the 5-lemma in each degree of homotopy proves that $\phi$ is a homotopy equivalence. Notice also that the commutativity of the above diagram indicates that this is an equivalence of $\E_0$-$X(n)$-modules. 

It follows then that the free $\E_1$-algebras on the equivalent $\E_0$-algebras $cof(\bar{\alpha})$ and $cof(\bar{\chi}_{n,p})$ are equivalent. Hence, as described above, \cite[{Lemma 4.4}]{acb} implies that $$X(n)_{(p)}\modd_{\E_1}\alpha\simeq F_{\E_0}^{\E_1}(cof(\bar{\alpha}))\simeq X(n+1)_{(p)}.$$
\end{proof}

\begin{rem}
We note that all of the above, and the results of \cite{beardsrelative}, only succeed in producing $X(n+1)$ as an $\E_1$-$X(n)$-algebra. However, $X(n+1)$ is an $\E_2$-$\mathbb{S}$-algebra (though unpublished work of Lawson indicates that $X(n+1)$ is provably not $\E_3$), and this structure is not reflected in this work.
	
%It should be noted that we have \emph{not} defined an algorithm for constructing $MU_{(p)}$ from $\sph_{(p)}$. The obstruction to doing so is the following: to attach an $\E_k$-$A$-cell to a ring spectrum $A$, we must have that $A$ is an $\E_{k+1}$-ring. The reason for this is that the cell attachment occurs in the category of $A$-algebras, which are contained in the category of $A$-modules, which is one step \emph{less} monoidal than $A$. In other words, if $A$ is an $\E_k$-ring spectrum then its category of modules is $\E_{k-1}$-monoidal, and so we can only attach $\E_{k-1}$-cells. 
%
%In our case, the construction of $X(n)$ as a Thom spectrum over a double loop space implies that it is an $\E_2$-algebra, so we can always attach $\E_1$-cells. And indeed, what we have shown above is that whenever we attach an $\E_1$-cell to $X(n)_{(p)}$ along any generator of $\pi_{2n-1}$, we obtain a spectrum which is equivalent to $X(n+1)_{(p)}$, but only as an $\E_1$-algebra. As such, we cannot repeat the cell attachment process without already knowing about the $\E_2$-structure of $X(n+1)$. 
%
%A potential solution to this problem might begin with showing that the $X(n)$-algebra unit map $X(n)\to X(n+1)$, as a map of $\E_2$-algebras, is \emph{central} for all $n$, in the sense of Warning 7.1.3.9 of \cite{ha}. This would at least require the computation of the relative $\E_2$-Hochschild cohomology of $X(n+1)$ under $X(n)$. 
\end{rem}

\begin{rem}
Recalling that $X(1)\simeq \sph$, $\chi_1\simeq \eta$ and $cof(\eta)\simeq\Sigma^{-2}\mathbb{C}P^2$, we see that $X(2)$ is equivalent to the free $\E_1$-ring spectrum on the $\E_0$-spectrum $\Sigma^{-2}\mathbb{C}P^2$ (whose $\E_0$ structure is provided by the inclusion of the bottom cell $\sph\to\Sigma^{-2}\mathbb{C}P^2$). This should be compared to Hopkins' description of $X(2)_{(2)}$ in \cite{hopthesis}.
\end{rem}

\begin{rem}
We have from \cite{rav} that $H_\ast(X(n)_{(p)})\cong \mathbb{Z}_{(p)}[b_2,\ldots,b_{2n-2}]$ for all $n$. If there is no $m$ such that $n= p^m-1$, $X(n+1)_{(p)}$ is equivalent to an infinite wedge of suspended copies of $X(n)_{(p)}$. In particular, one has a copy of $X(n)_{(p)}$ for each power of the polynomial generator $b_n\in \pi_{2n-2}(X(n)_{(p)})$ \cite{ravprivate} (this can also be deduced from \cite[Chapter 6]{rav}). In other words, the fact that $H_\ast(X(n+1)_{(p)})$ is a polynomial algebra over $H_\ast(X(n)_{(p)})$ is manifested at the level of spectra whenever passing from $n$ to $n+1$ does not increase the largest power of $p$ such that $p^m\leq n$. 

This indicates that the $p$-localization of the element $\chi_n$, and hence the homotopy group it generates, is zero when $n$ is not one less than a power of a $p$. Theorem \ref{thm:canonical} implies that for $n=p^m-1$ for some $m$, every other generator of $\pi_{2n-1}(X(n))$ is of the form $r\chi_n$ for an integer $r$ which is not divisible by $p$. So in this case we can conclude that $\pi_{2n-1}(X(n))$ must have been $\ints/p^e$ for some $e$. We leave the detailed investigation of this structure, and these homotopy groups, to a future paper.
\end{rem}

\begin{rem}
	From \cite[Example 1.22]{ravinfdesc}, we can see that $\chi_{p-1}$ becomes (a sum of copies of) $\alpha_1$ when localized at $p$. It seems likely that the $p$-localization of $\chi_{p^m-1}$ recovers the $m^{th}$ Greek letter element of the $p$-localized stable homotopy groups of spheres.
\end{rem}

\section{$\E_k$-monoidal Analogues of $X(n)$}\label{section:xnk}

We now introduce a new sequence of spectra which are $\E_k$-monoidal, have colimit equivalent to $MU$, and which are (eventually) produced by attaching multiplicative cells. We generalize Corollary \ref{generators} and Theorem \ref{thm:canonical} to these new spectra.

\begin{defi}
For an odd integer $k$, let $X(n,k)$ be the sequence of spectra obtained by Thomifying (the zero components of) the sequence of $k+1$-fold loop spaces $$\ast\simeq\Omega^kSU(1)\to\Omega^kSU(2)\to\ldots\to\Omega^kSU(n)\to\ldots\to\Omega^kSU\simeq \mathbb{Z}\times BU.$$
\end{defi} 

\begin{rem}
From \cite[Theorem 2.8]{acb} we can deduce that the spectra $X(n,k)$ are $\E_{k+1}$-monoidal ring spectra. This is a classical result of Lewis \cite{lewis}, and is also shown specifically for Thom spectra in \cite[Theorem 1.7]{abg}. 
\end{rem}

\begin{thm}\label{chiks}
For $n\geq\frac{k-1}{2}$ there is an element $\chi_{n,k}\in\pi_{2n-k}(X(n,k))$ such that the spectrum $X(n+1,k)$ is equivalent to $X(n,k)$ with an $\E_k$-$X(n,k)$-cell attached along $\chi_{n,k}$. 
\end{thm}

\begin{proof}
Fix $n$ and $k$ and assume $n\geq \frac{k-1}{2}$. Recall that there is a fiber sequence $SU(n)\to SU(n+1)\to S^{2n+1}$. By applying the based loops functor $k$ times we obtain a fiber sequence $\Omega^kSU(n)\to \Omega^k SU(n+1)\to \Omega^kS^{2n+1}$. If $2n+1\geq k$ then the base space of the preceding fibration is equivalent to the free $\E_k$-algebra on $S^{2n+1-k}$. By \cite[Theorem 1]{beardsrelative}  there is a morphism of $\E_{k}$-spaces $\phi\colon\Omega^k\Sigma^kS^{2n+1-k}\to BGL_1(X(n,k))$ whose associated $X(n,k)$-algebra Thom spectrum is equivalent to $X(n+1,k)$ as an $\E_k$-$X(n,k)$-algebra. 

Since the domain of the map $\phi$ is the free $\E_k$-algebra on $S^{2n+1-k}$, we have, by the free-forgetful adjunction, a map of based spaces $\phi'\colon S^{2n+1-k}\to BGL_1(X(n,k)).$ By the loops-suspension adjunction, and the fact that $GL_1(X(n,k))$ is equivalent to $\Omega BGL_1(X(n,k))$, we thus have a map of based spaces $\phi''\colon S^{2n-k}\to GL_1(X(n,k))$. Note that, since $GL_1(X(n,k))$ is composed of the connected components of the units in $\pi_0(\Omega^\infty X(n,k))\cong \ints$, the inclusion $i\colon GL_1(X(n,k))\hookrightarrow \Omega^\infty X(n,k)$ is not based. In particular, the base point of $GL_1(X(n,k))$ corresponds to 1 and the base point of $\Omega^\infty X(n,k)$ corresponds to 0. Let $u=\pm 1$ be the connected component containing the image of $\phi''$. Then there is a translation map $\tau_u\colon \Omega^\infty X(n,k)\to \Omega^\infty X(n,k)$ such that the composition $\tau_u\circ i\circ \phi''$ is a based map $S^{2n-k}\to \Omega^\infty X(n,k)$. Finally, by the adjunction between $\Omega^\infty$ and $\Sigma^\infty$, we obtain a map of spectra $\sph^{2n-k}\to X(n,k)$. This defines our element $\chi_{n,k}\in\pi_{2n-k}(X(n,k))$. An application of \cite[Theorem 4.10]{acb} gives that $X(n+1,k)$, the Thom spectrum of the map $\phi$ that we began with, is equivalent to $X(n,k)\modd_{\E_k}\chi_{n,k}$. 
\end{proof}

\begin{cor}
	For $n\geq\frac{k-1}{2}$, there is a Thom isomorphism of $\E_k$-$X(n,k)$-algebras $$X(n+1,k)\wedge_{X(n,k)}X(n+1,k)\simeq X(n+1,k)\wedge_\sph \Omega^k\Sigma^k S^{2n+1-k}.$$ 
\end{cor}

\begin{proof}
	This is precisely the Thom isomorphism of \cite[Theorem 8.4]{abg} used along with \cite[Corollary 3.17]{acb}. In particular, since $X(n+1,k)$ is a Thom spectrum associated to a map of $\E_k$-spaces $\Omega^k\Sigma^kS^{2n+1-k}\to BGL_1(X(n,k))$ it is oriented with respect to itself, so admits a Thom isomorphism. 
\end{proof}

We emulate Corollary \ref{generators} and Theorem \ref{thm:canonical} below as Corollaries \ref{cor:emu1} and \ref{cor:emu2} after confirming a fact about $X(n,k)$ that was shown for $X(n)=X(n,1)$ in \cite{rav}.

\begin{lem}\label{lem:equivbelow}
For odd $k$, the map $X(n,k)\to MU$ induced by the inclusion $\Omega^k SU(n)\hookrightarrow BU$ is an isomorphism on homotopy groups in degrees less than or equal to $2n-k-1$. 
\end{lem}

\begin{proof}
First recall that the inclusion $SU(n)\hookrightarrow SU$ is an equivalence in degrees less than $2n$. It follows then that, after applying $\Omega^k$, for odd $k$, $\Omega^k SU(n)\to BU$ is an equivalence in degrees less than $2n-k$. 

Now recall that the Thom spectrum functor (being computable as a colimit itself as in \cite{abghr} or as a left adjoint as in \cite{abg}) preserves colimits. Thus the cofiber of the map $Th(i)\colon X(n,k)\to MU$ is the Thom spectrum of the cofiber of $i\colon\Omega^k SU(n)\to BU$ (taken in the category of spaces over $BU$). Note that cofibers in overcategories are computed in the underlying categories, so the homotopy groups of $cof(i)$ as a space are isomorphic to the homotopy groups of $cof(i)$ as a space \emph{over $BU$}. Clearly the space $cof(i)$ has trivial homotopy groups in degrees less than $2n-k$. 

Consider the pointed map $BU\to BGL_1(\sph)\to BGL_1(H\ints)$ given by the complex $J$-homomorphism followed by the delooping of the canonical map $GL_1(\sph)\to GL_1(H\ints)$. This map is null because $BGL_1(H\ints)\simeq \mathbb{R}P^\infty$ and $H^1(BU;\ints/2)=0$. So we have, via the orientation theory of \cite{abghr}, the following three equivalences: $H\ints\wedge X(n,k)\simeq H\ints\wedge \Sigma^\infty_+\Omega^k SU(n)$, $H\ints \wedge MU\simeq H\ints\wedge \Sigma^\infty_+BU$ and $H\ints \wedge Th(cof(i))\simeq H\ints\wedge \Sigma^\infty_+ cof(i)$. In other words, we have Thom isomorphisms in integral homology for all three Thom spectra. Since we know that $cof(i)$ must have trivial homotopy groups in degrees less than $2n-k$, it also has trivial homology in those degrees. Hence $Th(cof(i))$ also has trivial homology in those degrees and thus trivial homotopy there as well (as it is connective). So the Thom spectrum functor preserves connectivity of the map $\Omega^kSU(n)\to BU$, causing the map $X(n,k)\to MU$ to be an equivalence in degrees less than $2n-k$.   
\end{proof}

\begin{rem}\label{hencetrivial}
It follows from the above, and the fact that $\pi_\ast(MU)$ is trivial in odd degrees, that $\pi_{2n-k}(X(n+1,k))=0$ because $2n-k$ is odd and $2n-k<2(n+1)-k$. 
\end{rem}

\begin{cor}\label{cor:emu1}
The elements $\chi_{n,k}$ described in Theorem \ref{chiks} generate $\pi_{2n-k}(X(n,k))$ as abelian groups.
\end{cor}

\begin{proof}
Since $X(n+1,k)\simeq X(n,k)\modd_{\E_k}\chi_{n,k}$, this follows immediately from Theorem \ref{survive} and Remark \ref{hencetrivial}. We are also using the fact that $\pi_0(X(n,k))\cong\ints$. Note that this corollary is entirely analogous to Corollary \ref{generators}.
\end{proof}

\begin{cor}\label{cor:emu2}
If $p\in\ints$ is a prime and $n\geq \frac{k-1}{2}$ then $X(n+1,k)_{(p)}$ is equivalent to $X(n,k)_{(p)}$ with an $\E_k$-cell attached along any generator of $\pi_{2n-k}(X(n,k)_{(p)}$. 
\end{cor}

\begin{proof}
The proof proceeds completely identically to the proof of Theorem \ref{thm:canonical}. 
\end{proof}

\begin{rem}
Recall that the spectra $X(n,1)=X(n)$ have an interpretation in terms of truncated formal group laws, or \emph{formal $n$-buds}, as described in \cite[\S 6.5]{rav}. The author is not aware of any similar interpretation of the spectra $X(n,k)$ for $k>1$. Such an interpretation, even for $k=3$ in light of the appearance of the space $\Omega^3 SU(n)$ in the proof of the Nilpotence Conjecture, would likely be enlightening. 
\end{rem}

\appendix

\section{Some Categorical Constructions}\label{appendix}

This appendix contains several category theoretic results that are used in the main body of the paper.  They will not be surprising to experts, but rigorous proofs of them do not seem to be in the literature, so we include them here. First we describe a version of the bar construction for computing pushouts in quasicategories. Then we investigate the ``symmetric powers'' filtration on the free $\E_k$-algebra on a module over an $\E_{k+1}$-ring spectrum.

\begin{lem}\label{barkan}
	Let $K=N(a\overset{\ell}\leftarrow b \overset{r}\to c)$ be the nerve of the span category and $F\colon K\to \C$ be a pushout diagram in a cocomplete quasicategory $\C$ with $F(a)=A$, $F(b)=B$ and $F(c)=C$. Let $i\colon K\to N(\Delta^{op})$ be the map of quasicategories induced by the functor $Span\to\Delta^{op}$ taking $a$ and $c$ to $[0]$, $b$ to $[1]$ and the morphisms to $d_0$ and $d_1$.  Then the quasicategorical left Kan extension of $F$ along $i$, denoted $Lan_i(F)$, has the property that $Lan_i(F)([n])\simeq A\coprod B^{\coprod n}\coprod C$.  
\end{lem}

\begin{proof}
	Since $\C$ is cocomplete, the left Kan extension exists. It follows from \cite[Proposition 5.2.4]{rv2} that Kan extensions are stable under pasting with comma squares. Thus given an object $x\colon\Delta^0\hookrightarrow N(\Delta^{op})$ we have a commutative diagram:
	
	$$\xymatrix{
		i\downarrow x\ar[d]\ar[r]&K\ar[d]^{i}\ar[r]^{F}&\mathcal{C}\\
		\Delta^0\ar[r]^x& N(\Delta^{op})\ar@{-->}[ur]_{Lan_i(F)}& 
	}$$ 
	
	In the above diagram, the square is a pullback (or comma) square and stability under pullback squares means that on the object $x\in N(\Delta^{op})$, the value of $Lan_i(F)$ can be computed as the colimit over the comma category $i\downarrow x$ of the composition of the forgetful functor $U\colon i\downarrow x\to K$ followed by $F$. In other words, on each object, $Lan_i(F)$ can be computed as the left Kan extension of the composition ${i\downarrow x}\to K\overset{F}\to \mathcal{C}$ along the map to the final object $\Delta^0$. This is essentially the classical ``pointwise'' description of left Kan extensions. Objects of the category $i\downarrow x$ are objects of $N(\Delta^{op})$ over $x$ in $N(\Delta^{op})$ whose domains are $i(y)$ for some $y\in K$. We will denote them by $i(y)=[j]\to x$ for $y\in\{a,b,c\}$ and $j\in\{0,1\}$. 
	
	To see that this gives the claimed description of $Lan_i(F)([n])$, we consider the comma category $i\downarrow [n]$ at the level of objects and 1-morphisms. This suffices because our description of $Lan_i(F)([n])$ will be as a colimit over a diagram in the nerve of an ordinary category. Note that there is a unique degeneracy map $[0]\to [n]$ in $\Delta^{op}$ corresponding to the unique surjection $[n]\to [0]$, and that there are $n$ unique degeneracies $[1]\to [n]$ in $\Delta^{op}$ corresponding to the $n$ order-preserving surjections $[n]\to [1]$ in $\Delta$. There cannot be any 1-morphisms in $i\downarrow[n]$ between the $n$ objects $i(b)=[1]\to [n]$ because this would require the existence of commutative diagrams implying that some of the $n$ unique surjections in $\Delta$ be equal. There are however two \textit{non-surjective} order preserving maps $[n]\to[1]$ in $\Delta$ corresponding to the maps that take all the elements of $\{0,1,\ldots,n\}$ to either $0$ or $1$. However, because (the opposites of) the maps $[0]\to[1]$ are in the image of $i$, these two non-surjective functions admit 1-morphisms in $i\downarrow [n]$ to the unique morphisms $i(a)=[0]\to[n]$ and $i(c)=[0]\to [n]$. This exhausts all the maps $[n]\to[1]$. Thus, after applying the forgetful functor $i\downarrow x\to K$ and then applying $F$, we have a diagram which is in fact just a coproduct of diagrams: 
	\[
	F(a\leftarrow b)\coprod \underbrace{F(b)\coprod\cdots\coprod F(b)}_{n~\text{copies}}\coprod F(b\to c).
	\]
	The colimit of this diagram in $\mathcal{C}$ clearly has the desired form.
\end{proof}

The above lemma may require some explanation. Recall the classical bar construction $B_\bullet(M,A,N)$ where $M$ and $N$ are left and right modules over some algebra $A$, respectively (cf.~Sections 4.1 and 4.2 of \cite{riehlhct}, especially Example 4.1.2). It is a simplicial object with $B_n(M,A,N)=M\otimes A^{n}\otimes N$ whose face maps are given by the multiplication on $A$ and the $A$ actions on $M$ and $N$. In a model category, the bar construction is a way of computing a left derived tensor product $M\otimes^L_A N$.

A span in a cocomplete category $a\overset{r}\leftarrow b\overset{\ell}\to c$ induces a left $b$-module structure on $a$ and right $b$-module structure on $c$ with respect to the cocartesian monoidal structure. To see this, first notice that every object of such a category is an algebra with respect to this monoidal structure since there is always a ``fold'' map $x\coprod x\overset{\phi}\to x$. Then note that the maps $r$ and $\ell$ of the pushout induce maps $a\coprod b\overset{1_a\coprod r}\to a\coprod a\overset{\phi}\to a$ and $b\coprod c\overset{\ell\coprod 1_c}\to c\coprod c\overset{\phi}\to c$. So the colimit of the bar construction in this setting, $B_\bullet(a,b,c)$, is a model for the \emph{homotopy} pushout of the diagram $a\overset{r}\leftarrow b\overset{\ell}\to c$. 

Thus Lemma \ref{barkan}, using the fact that the colimit of a left Kan extension of $F$ agrees with the colimit of $F$, can be thought of as showing that a pushout in a quasicategory can be computed by taking the colimit of a suitable bar construction $B_\bullet(F(a),F(b),F(c))$. In other words, Lemma \ref{barkan} is just a rephrasing of a standard homotopy theoretical fact in the language of quasicategories.

Now recall from \cite[Construction 3.1.3.9]{ha} and \cite[Construction 3.1.3.13]{ha} that the free $\E_k$-algebra on an object $M$ of a quasicategory $\mathcal{C}$ admits a decomposition as a coproduct in $\mathcal{C}$ of the form $\coprod_{n\geq 0}Sym_{\E_k}^n(M)$. We will describe the objects $Sym_{\E_k}^n$ as certain colimits in the proof of Lemma \ref{lemma:symconnected}. However, if $\mathcal{O}$ happens to be the $\infty$-operad associated to a classical simplicial colored operad $\mathcal{O}_0$ then $Sym^n_{\mathcal{O}}(M)\simeq (A[\mathcal{O}_0(n)]\otimes_A M^{\otimes_A n})_{\Sigma_n}$. As such, one recovers the classical symmetric power decomposition of the free $\mathcal{O}$-algebra on an object. We will not prove this here but refer the interested reader to \cite[Corollary 3.2.7]{GHK}.

\begin{lem}\label{lemma:symconnected}
	Let $A$ be an $\E_{k+1}$ ring spectrum and $M$ be an $A$-module which is $d$-connective for some $d\geq 0$. Then $Sym^n_{\E_k}(M)$ is $nd$-connective for all $n>1$. 	
\end{lem}

\begin{proof}
	The proof is nothing more than a careful interpretation of \cite[Construction 3.1.3.9]{ha}. We will attempt to transcribe such an interpretation here, where we have replaced $\mathcal{O}^\otimes$ with $\E_k^\otimes$, $\mathcal{C}^\otimes$ with $LMod_A^\otimes$, $C$ with $M$, and $X$ and $Y$ with $\langle 1\rangle$. Recall that the object $Sym_{\E_k}^n(M)$ is constructed as a colimit of a functor $\bar{h}_1\colon \mathcal{P}(n)\to LMod_A$, where, in this case, $\mathcal{P}(n)$ is the full subcategory of $\mathcal{T}riv^\otimes\times_{\E_k}(\E_k/{\langle 1\rangle})$ spanned by the $(\langle n\rangle,\alpha\colon\langle n\rangle\to \langle 1\rangle)$, with $\alpha$ an \emph{active} morphism. Recall that $\mathcal{T}riv^\otimes$ is the \textit{trivial $\infty$-operad} of \cite[Example 2.1.1.20]{ha}, and active here means that $\alpha$ is a map of finite pointed sets with $\alpha^{-1}(\ast)=\{\ast\}$. The 1-simplices of $\mathcal{P}(n)$ are forced to be precisely the pointed bijections $\langle n\rangle\to \langle n\rangle$, so it is a Kan complex. We can project $\mathcal{P}(n)$ onto $\E_k^\otimes$ by the composition $$\mathcal{P}(n)\hookrightarrow \mathcal{T}riv^\otimes\times_{\E_k}(\E_k/{\langle 1\rangle})\twoheadrightarrow \E_k/\langle 1 \rangle\twoheadrightarrow \E_k,$$ and this admits a natural transformation (since it factors through the slice over $\langle 1\rangle$) to the constant functor $h_1\colon \mathcal{P}(n)\to\E_k$ valued at $\langle 1\rangle$. This corresponds to a map $h\colon \mathcal{P}(n)\times\Delta^1\to \E_k$. 
	
	Now because there is (by the definition of $\infty$-operad) a coCartesian structure map $q\colon LMod_A^\otimes\to\E_k^\otimes$, we can lift $h$ to a map $\bar{h}:\mathcal{P}(n)\times\Delta^1\to LMod_A^\otimes$. Then $\bar{h}_1$ is defined as the restriction of $\bar{h}$ to $\mathcal{P}(n)\times\{1\}$. The lift is coCartesian over the active map from $\langle n\rangle$ to $\langle 1\rangle$, so it must be the monoidal structure map that takes an $n$-tuple $\{M_1,\ldots,M_n\}\in q^{-1}(\langle n\rangle)\simeq LMod_A^{n}$ to $M_1\wedge_A\cdots\wedge_A M_n\in LMod_A$. In this case, because we are forming the free $\E_k$-algebra on $M$, the image of the unique (up to homotopy) object of $\mathcal{P}(n)\times\{0\}$ in $LMod_A^\otimes $ is, by construction, equivalent to $\{M,M,\ldots,M\}\in LMod_A^n$. This means that the image of $\bar{h}_1$ is equivalent to $M^{\wedge_A n}$. 
	
	Because $M$ is $d$-connective we have that $M^{\wedge_A n}$ is $nd$-connective. By \cite[Corollary 4.2.3.5]{ha} colimits of $A$-modules are constructed on underlying spectra. Then since by \cite[Proposition 4.4.2.6]{ha} any colimit of spectra can be decomposed into coproducts and pushouts, and coproducts and pushouts both preserve connectivity (which can be seen by using, for instance, the Mayer-Vietoris sequence), we know that colimits of $A$-modules preserve connectivity. Hence we have that the colimit of $\bar{h}_1$, which is the definition of $Sym_{\E_k}^n(M)$, is also $nd$-connective. 
\end{proof}

\begin{lem}\label{lemma:tpcommuteswithsifted}
	The functors $Sym_{\E_k}^n\colon LMod_A\to LMod_A$ commute with sifted colimits. 
\end{lem}

\begin{proof}
	By the proof of the previous lemma, these functors decompose as first forming the $n$-fold tensor product of a module $M$ and then taking a colimit of a diagram whose vertices are all equivalent to $M^{\wedge_A n}$. Taking colimits preserves colimits, so it is only required to show that forming $n$-fold tensor powers commutes with sifted colimits. 
	
	First recall that for any diagram $F\colon D\to LMod(A)$, the tensor product preserves colimits in each variable. In other words, there is an equivalence $colim_D(F(d))\wedge_A M\simeq colim_D(F(d)\wedge_AM)$.  From this fact one can inductively deduce, for any $n$,
	\[
	colim_D(F(d_1))\wedge_A\cdots\wedge_A colim_D(F(d_n))\simeq colim_{D^n}(F(d_1)\wedge_A\cdots\wedge_AF(d_n)).
	\] 
	
	Now recall that $D$ is sifted if and only if the diagonal map $\delta\colon D\to D\times D$ is cofinal, which implies that the $n$-fold diagonal $\delta_n\colon D\to D^n$ is also cofinal. That is to say, a colimit over $D^n$ may be equivalently computed by pulling back to $D$ along $\delta_n$. Thus it follows that
	\[
	colim_{D^n}(F(d_1)\wedge_A\cdots\wedge_AF(d_n))\simeq \colim_D(F(d)^{\wedge n}),
	\]
	which completes the proof.
\end{proof}

\bibliographystyle{abbrv}
\bibliography{references}

\begin{thebibliography}{10}

\bibitem{abg}
M.~Ando, A.~J. Blumberg, and D.~Gepner.
\newblock Parametrized spectra, multiplicative {T}hom spectra, and the twisted
  {U}mkehr map, 2015.
\newblock arxiv.org/abs/1112.2203.

\bibitem{abghr}
M.~Ando, A.~J. Blumberg, D.~Gepner, M.~J. Hopkins, and C.~Rezk.
\newblock An {$\infty$}-categorical approach to {$R$}-line bundles,
  {$R$}-module {T}hom spectra, and twisted {$R$}-homology.
\newblock {\em J. Topol.}, 7(3):869--893, 2014.

\bibitem{acb}
O.~Antol\'in-Camarena and T.~Barthel.
\newblock A simple universal property of {T}hom ring spectra, 2014.
\newblock arxiv.org/abs/1411.7988.

\bibitem{batwhite}
M.~Batanin and D.~White.
\newblock Bousfield localization and {E}ilenberg-{M}oore categories, 2016.
\newblock arXiv:1606.01537.

\bibitem{beardsrelative}
J.~Beardsley.
\newblock Relative {T}hom spectra via operadic {K}an extensions.
\newblock {\em Algebr. Geom. Topol.}, 17(2):1151--1162, 2017.

\bibitem{devhopsmith}
E.~S. Devinatz, M.~J. Hopkins, and J.~Smith.
\newblock Nilpotence and stable homotopy theory {I}.
\newblock {\em Annals of Mathematics}, 128(2):207--241, 1988.

\bibitem{GHK}
D.~Gepner, R.~Haugseng, and J.~Kock.
\newblock $\infty$-operads as analytic monads, 2017.
\newblock arXiv:1712.06469.

\bibitem{hopthesis}
M.~J. Hopkins.
\newblock {\em Stable decompositions of certain loop spaces}.
\newblock PhD thesis, Northwestern University, 1984.

\bibitem{laz}
M.~Lazard.
\newblock {\em Commutative formal groups}.
\newblock Lecture Notes in Mathematics, Vol. 443. Springer-Verlag, Berlin-New
  York, 1975.

\bibitem{lewis}
L.~Lewis.
\newblock {\em The Stable Category and Generalized {T}hom Spectra}.
\newblock PhD thesis, The University of Chicago, 1978.

\bibitem{htt}
J.~Lurie.
\newblock {\em Higher topos theory}, volume 170 of {\em Annals of Mathematics
  Studies}.
\newblock Princeton University Press, Princeton, NJ, 2009.

\bibitem{ha}
J.~Lurie.
\newblock Higher algebra, 2014.
\newblock math.harvard.edu/\textasciitilde lurie/papers/higheralgebra.pdf.

\bibitem{nakairavenel}
H.~Nakai and D.~C. Ravenel.
\newblock The method of infinite descent in stable homotopy theory ii, 2018.
\newblock arXiv:1806.10892.

\bibitem{priddybp}
S.~Priddy.
\newblock A cellular construction of {BP} and other irreducible spectra.
\newblock {\em Math. Z.}, 173(1):29--34, 1980.

\bibitem{rav}
D.~Ravenel.
\newblock {\em Complex Cobordism and the Homotopy Groups of Spheres}.
\newblock Academic Press, 1986.

\bibitem{ravinfdesc}
D.~Ravenel.
\newblock The method of infinite descent in stable homotopy theory. {I}.
\newblock In {\em Recent progress in homotopy theory ({B}altimore, {MD},
  2000)}, volume 293 of {\em Contemp. Math.}, pages 251--284. Amer. Math. Soc.,
  Providence, RI, 2002.

\bibitem{ravprivate}
D.~Ravenel.
\newblock Private e-mail communication, 2018.

\bibitem{riehlhct}
E.~Riehl.
\newblock {\em Categorical Homotopy Theory}.
\newblock Number~24 in New Mathematical Monographs. Cambridge University Press,
  2014.

\bibitem{rv2}
E.~Riehl and D.~Verity.
\newblock Kan extensions and the calculus of modules for {$\infty$}-categories.
\newblock {\em Algebr. Geom. Topol.}, 17(1):189--271, 2017.

\bibitem{szymikprimechar}
M.~Szymik.
\newblock Commutative {$\Bbb S$}-algebras of prime characteristics and
  applications to unoriented bordism.
\newblock {\em Algebr. Geom. Topol.}, 14(6):3717--3743, 2014.

\end{thebibliography}

\end{document}